\documentclass[12pt,reqno]{amsart}
\advance \topmargin by -\headheight
\advance \topmargin by -\headsep
\evensidemargin \oddsidemargin
\usepackage{amsmath}
\usepackage{bbm}
\usepackage{amsfonts}
\usepackage{stmaryrd}
\usepackage{amssymb}
\usepackage{amsthm}
\usepackage{csquotes}
\usepackage{color}
\usepackage{mathrsfs}
\usepackage{dsfont}
\usepackage{bm}
\usepackage{cite}
\usepackage{enumerate}
\usepackage{booktabs}

\numberwithin{equation}{section}

\newtheorem{theorem}{Theorem}[section]
\newtheorem{corollary}[theorem]{Corollary}
\newtheorem{lemma}[theorem]{Lemma}
\newtheorem{proposition}[theorem]{Proposition}
\newtheorem{fact}[theorem]{Fact}

\theoremstyle{definition}
\newtheorem{definition}[theorem]{Definition}
\newtheorem{example}[theorem]{Example}
\theoremstyle{remark}

\newcommand{\R}{\mathbb R}

\newcommand{\Z}{\mathbb Z}
\newcommand{\Per}{\mathrm{Per}}
\def\d{\,\mathrm{d}}

\newcommand{\BM}{\mathrm{BM}}
\newcommand{\M}{\mathcal{M}}
\newcommand{\ndim}{\mathrm n}
\newcommand{\hdim}{\mathrm h}

\newcommand{\vol}{\operatorname{vol}}

\newcommand{\PSL}{\operatorname{PSL}}

\usepackage{mathtools}

\title{A nonabelian Brunn--Minkowski inequality II}
\author{Yifan Jing}
\address{Department of Mathematics, The Ohio State University, Columbus OH, 43210, USA}
\email{jing.245@osu.edu}

\author{Chieu-Minh Tran}
\address{Department of Mathematics, National University of Singapore, 119077, Singapore}
\email{trancm@nus.edu.sg}

\begin{document}
\begin{abstract}
We prove that every unimodular locally compact group $G$ of noncompact Lie dimension $n$ satisfies the sharp Brunn--Minkowski inequality
\[
  \mu_G(XY)^{1/n}\ge\mu_G(X)^{1/n}+\mu_G(Y)^{1/n},
\]
and establish a general form for arbitrary, possibly nonunimodular, locally compact groups. This fully confirms the nonabelian Brunn--Minkowski conjecture proposed by the present authors and Zhang.
As an application, we obtain an isoperimetric inequality on symmetric spaces of noncompact type. 
\end{abstract}
\maketitle

\section{Introduction}\label{sec:introduction}

Let $G$ be a locally compact group, let $\mu_G$ be a left Haar measure, and let $\nu_G$ be a right Haar measure.  For subsets $X,Y\subseteq G$ we write $XY=\{xy:x\in X,\ y\in Y\}$ for their product set, using the additive notation $X+Y$ when $G$ is abelian.  When $G$ is the Euclidean space $\R^d$, the Brunn--Minkowski inequality states that
\[  
\lambda(X+Y)^{1/d}\ge\lambda(X)^{1/d}+\lambda(Y)^{1/d}
\]
for all nonempty compact $X,Y\subseteq\R^d$, where $\lambda$ denotes the usual Lebesgue measure.  The inequality is central to convex geometry and has applications in many other areas; see Gardner's survey~\cite{Gardner}.

What survives when $\R^d$ is replaced by a noncommutative group is a question of long standing.  It goes back to the measure theoretic treatment of sumsets by Henstock and Macbeath~\cite{HenstockMacbeath}, who proposed the problem of generalizing the Brunn--Minkowski inequality to general locally compact groups. Different variations of this problem were also later suggested by Hrushovski~\cite{HrushovskiPC}, by McCrudden~\cite{McCrudden}, and by Tao~\cite{TaoBlog}. Partial results for nilpotent and solvable groups were obtained by Bobkov~\cite{Bobkov}, by Gromov~\cite{Gromov}, by Hrushovski~\cite{HrushovskiLieModel}, by Leonardi and Masnou~\cite{LeoMasnou}, by McCrudden~\cite{McCrudden}, by Pozuelo~\cite{Pozuelo}, and by Tao~\cite{TaoBlog}. 

The general form of the Brunn--Minkowski inequality was found only recently, when the present authors and Zhang attached to $G$ two integer invariants, its noncompact Lie dimension $\ndim(G)$ and its helix dimension $\hdim(G)$.  For a Lie group of dimension $d$ one sets
\begin{equation}\label{eq:noncompact-dim}
  \ndim(G)=d-\max\{\dim K:K<G\text{ is compact}\}.
\end{equation}
Thus $\ndim(G)$ is the codimension of a maximal compact subgroup, and for connected $G$ it vanishes precisely when $G$ is compact.  The helix dimension is a smaller companion invariant that detects a helical obstruction carried by the center. For a connected simple group it equals the rank of the discrete center and never exceeds one. Although we have introduced the two invariants just for Lie groups, both extend to an arbitrary locally compact group: the Gleason--Yamabe solution~\cite{Gleason,Yamabe} to Hilbert's fifth problem approximates $G$ by Lie groups, and $\ndim(G)$ and $\hdim(G)$ are read off from these approximations; we refer to \cite{JTZ} for the details. 

The main theorem of \cite{JTZ} establishes the inequality with exponent $1/(\ndim(G)-\hdim(G))$, sharp whenever $\hdim(G)=0$.  A positive helix dimension, by contrast, brings a genuine loss, which \cite{JTZ} conjectured to be an artifact of the method rather than a feature of the geometry.

Our main theorem removes the loss in complete
generality.

\begin{theorem}\label{thm:all-lc}
Suppose $G$ is a locally compact group with $\ndim(G)=n\ge0$.  Then for all
nonempty compact $X,Y\subset G$ of positive Haar measure,
\begin{equation}\label{eq:general-main}
 \left(\frac{\nu_G(X)}{\nu_G(XY)}\right)^{1/n}
 +
 \left(\frac{\mu_G(Y)}{\mu_G(XY)}\right)^{1/n}
 \le 1,
\end{equation}
where for $n=0$ the left-hand side is read as the maximum of the two ratios.
\end{theorem}

Inequality \eqref{eq:general-main} is the nonabelian Brunn--Minkowski conjecture proposed in~\cite{JTZ}, which the present paper thereby verifies. The helix dimension that measured the loss there plays no role in the sharp inequality, and only the noncompact Lie dimension $n$ survives. In particular, when $G$ is unimodular we prove the more familiar sharp form
\begin{equation}\label{eq:unimodular-main}
  \mu_G(XY)^{1/n}\ge \mu_G(X)^{1/n}+\mu_G(Y)^{1/n}.
\end{equation}

Several applications are discussed in Section 9. We derive a parallel volume inequality on homogeneous quotients, and specialize it to symmetric spaces of noncompact type. In the latter setting this yields the Euclidean Cartan--Hadamard isoperimetric lower bound. By classical symmetrization, this yields the corresponding Sobolev and Faber--Krahn inequalities.

\subsection*{Overview of the proof}
Two results of \cite{JTZ} organize the proof.  The reduction theorem there shows that it is enough to prove 
\eqref{eq:general-main} for Lie groups, and \eqref{eq:general-main} for an arbitrary Lie group follows once it is known for its simple factors.  The helix-free theorem resolves all simple groups with finite center.
What remains are the simply connected simple groups of Hermitian type, those with infinite center, whose Lie algebras are listed in \eqref{eq:hermitian-list}. This paper establishes \eqref{eq:unimodular-main} for such group, and Theorem~\ref{thm:all-lc} follows through the reduction just described.

Let us first briefly explain the method of ~\cite{JTZ} for a simple Lie group $G$ which we will need modify. The idea is based on considering the Iwasawa decompostion $KAN$ of $G$. They first prove a nonunimodular Brunn-Minkowski inequality for $AN$ with exponent $\ndim(AN)$ and move this to $G$ with unchanged exponent though the so-called proportionated fiber trick. This works perfectly when $K$ is compact and $\ndim(G)= \ndim(AN)$. For $G= \widetilde{\PSL}_2(\R)$, we have $K\simeq\R$ and $\ndim(AN)=2=\ndim(G)-1$, so we do not get the sharp exponent. A similar situation holds for other noncompact $G$.

We handle separately the unknown case of lowest dimension,  the universal cover $\widetilde{\PSL}_2(\R)$ of ${\PSL}_2(\R)$. This has noncompact Lie dimension three and helix dimension one. The result in \cite{JTZ} yields the exponent $1/2$, which we need to upgrade to the expected sharp value $1/3$ (Theorem~\ref{thm: PSL2}).  

The new ingredient in Theorem~\ref{thm: PSL2} is a monotone transport argument along the one dimensional factor in the Iwasawa decomposition.  We write
\[
  G=\widetilde{\PSL}_2(\R),\qquad
  H=AN,\qquad
  K=\{k_t:t\in\R\}\simeq\R.
\]
The two orders of the Iwasawa decomposition $KH$ and $HK$ give global coordinates along $K$. Using $G=KH$, we identify $G/H$ with $K\simeq\R$.  Under this identification, left multiplication by each $h\in H$ induces an increasing diffeomorphism
$
  \phi_h:\R\to\R
$
fixing the origin.  More precisely, if $q(k_th)=t$ denotes the $K$-coordinate in the $KH$-decomposition, then
\[
  q(k_t h k_s)=t+\phi_h(s).
\]
This order-preserving action is the geometric feature that makes monotone transport effective.

Let $X,Y\subset G$ be compact, and define their fibres using the two coordinate orders $KH$ and $HK$
\[
  C_t=\{c\in H:k_tc\in X\},
  \qquad
  D_s=\{d\in H:dk_s\in Y\}.
\]
Because $H$ is nonunimodular, the two fibre masses are measured using right and left Haar measure respectively, and denote them by $f(t)$ and $g(s)$. 

We then choose the nondecreasing transport map $T$ that sends the proportionated measure $f(t)\d t/\mu(X)$ to proportionated measure $g(s)\d s/\mu(Y)$. This is the same proportionated argument developed in $\PSL_2(\R)$ setting. 
At almost every relevant $t$, this gives the transport identity
\[
  T'(t)g(T(t))=\frac{\mu(Y)}{\mu(X)}f(t).
\]
For such $t$, set
$
  E_t=C_tD_{T(t)}\subseteq H.
$
For $h\in E_t$ consider the transported diagonal map
\[
  \Psi(t,h)=k_t h k_{T(t)}
\]
Its image is contained in $XY$.  The monotonicity of $T$, together with the order-preserving action of $H$ on $G/H$, implies that $\Psi$ is injective. The same quotient action also determines the Jacobian.  If
$
  \vartheta(h)=\varphi_h'(0),
$
then
\[
  J_\Psi(t,h)=1+T'(t)\vartheta(h).
\]
The area formula therefore yields
\[
  \mu_G(XY) \geq \int_\Omega(
    \nu_H(E_t)+T'(t)\mu_H(E_t))\d t,
\]
where $\Omega$ is a full fiber mass set and the transport identity holds on it.

Finally, we apply the asymmetric $\BM(2)$ inequality on the two dimensional solvable group $H$ to $C_t$ and $D_{T(t)}$. Combining this fiberwise estimate with the transport identity gives the $\BM(3)$ result for $G$. 
In summary, the one dimensional transport along $K$ and the two dimensional Brunn--Minkowski inequality inside $H$ combine to produce the sharp three dimensional exponent.

This transport construction is specific to $G= \mathfrak{sl}_2(\R)$ as it requires a factorization $G=KH$ with $K$ one-parameter and $H$ a closed connected subgroup of codimension one. For the remaining groups, the proof goes through a different direction. Let $d$ be the dimension of the group under consideration.  Using induction on dimension, we can assume we have proven the result for all connected simple Lie group with Hermitian type of dimension $< d(G)$. The reduction theorem and the helix-free theorem of~\cite{JTZ} tells us that we have also proven the result for all Lie groups of dimension $<d$. The role of the Iwasawa decomposition is now played by the contact parabolic decomposition $G=CP$ given in Proposition~\ref{prop:contact-factorization}. In this case with $G \neq \mathfrak{sl}_2(\R)$,  $C$ is a compact group of positive dimension, so the parabolic group $P$ has $\dim(P)<d$. By the consequence of the induction hypothesis described above, we have the statement for $P$, the proportionated fiber trick in~\cite{JTZ} then move this result to $G$ as desired.

\subsection*{Organization}
Section~\ref{sec:preliminaries} fixes the general Brunn--Minkowski notation and collects the structural results we borrow from \cite{JTZ}. The base case, Theorem~\ref{thm: PSL2}, occupies Sections~3--6. Sections~\ref{sec:contact} and \ref{sec:induction} set up the contact parabolic reduction and prove Theorem~\ref{thm:simple} and hence deduces Theorem~\ref{thm:all-lc}. Applications are discussed in Section~9. The type by type dimension computations behind the induction is carried out in Appendix~\ref{app:classification}.

\subsection*{Acknowledgements}
The authors would like to thank Ehud Hrushovski for introducing the authors to this problem, and Emmanuel Breuillard, Simon Machado, and Ruixiang Zhang for valuable discussions over the years. 
Part of the work was done when the first author was visiting the second author at the department of mathematics at National University of Singapore and he thanks the hospitality he received. 
The first author is partially supported by NSF grant DMS-2503063. 

\subsection*{Statement of AI} ChatGPT 5.6 is used for testing a few parallel ideas from the authors by doing computations on examples. It is also used for English grammar check and finding missing references. 
The method presented in the paper is from the authors and is inspired by these computation outcome. The authors are fully responsible for the mathematical content and the writing of the manuscript.

\section{Preliminaries}\label{sec:preliminaries}

In this section, we fix the notation used throughout the paper and record the structural results we borrow from~\cite{JTZ}.

\subsection{Haar measures and the general inequality}

Let $L$ be a locally compact group.  We write $\mu_L$ for a left Haar measure on $L$, $\nu_L$ for a right Haar measure, and $\Delta_L:L\to\R^{>0}$ for the modular function, so that
\[
  \mu_L(Eg)=\Delta_L(g)\mu_L(E)
\]
for every Borel set $E\subseteq L$ and every $g\in L$.  When $L$ is unimodular, we always take $\nu_L=\mu_L$.

For $r\in\Z^{\geq0}$ and $(x,y)\in\R^2$, as usual we set
\begin{equation}\label{eq:norm-notation}
\big\Vert (x,y)\big\Vert_{1/r}=
\begin{cases}
\big(|x|^{1/r}+|y|^{1/r}\big)^{r} &\text{ if } r\neq0,\\
\max\{|x|,|y|\} & \text{ if } r=0.
\end{cases}
\end{equation}

\begin{definition}[Brunn--Minkowski exponent]\label{def:BM}
Let $r\in\Z^{\geq0}$.  We say that $L$ satisfies the {\bf Brunn--Minkowski inequality with exponent $r$}, abbreviated as $\BM(r)$, if for all compact $X,Y\subseteq L$ of positive measure,
\[
\Bigg\Vert\Big(\frac{\nu_L(X)}{\nu_L(XY)},\frac{\mu_L(Y)}{\mu_L(XY)}\Big)\Bigg\Vert_{1/r}\leq1.
\]
\end{definition}

When $r>0$, the above says that
\begin{equation}\label{eq:BM-r}
 \left(\frac{\nu_L(X)}{\nu_L(XY)}\right)^{1/r}
 +
 \left(\frac{\mu_L(Y)}{\mu_L(XY)}\right)^{1/r}
 \leq1,
\end{equation}
and when $r=0$ the left hand side is replaced by the maximum of the two ratios.  When $L$ is unimodular and $r>0$, \eqref{eq:BM-r} is equivalent
to having the inequality
\[
  \mu_L(XY)^{1/r}\geq\mu_L(X)^{1/r}+\mu_L(Y)^{1/r}.
\]
Note that the definition does not depend on the normalization of the two Haar measures.

We also note that $\nu_L(X)\leq\nu_L(XY)$ and $\mu_L(Y)\leq\mu_L(XY)$. Indeed, $Xy\subseteq XY$ for every $y\in Y$ and $xY\subseteq XY$ for every $x\in X$, while $\nu_L$ is right invariant and $\mu_L$ is left invariant. This is the reason we pair $\nu_L$ with the left factor and $\mu_L$ with the right factor. Hence, both ratios in \eqref{eq:BM-r} are at most $1$, and in particular, every locally compact group satisfies $\BM(0)$. 
Moreover, it is clear that if $0\leq r<r'$ and $L$ satisfies $\BM(r')$, then $L$ satisfies $\BM(r)$. 

The following results from~\cite{JTZ} are stated here in the precise forms needed below.

\begin{fact}\label{fact:jtz}
Let $L$ be a Lie group and let $r,s\in\Z^{\geq0}$.
\begin{enumerate}[(1)]
 \item {\rm (Solvable case)} If $L$ is connected and solvable, then $L$ satisfies $\BM(\ndim(L))$.
 \item {\rm (Exponent splitting, \cite{JTZ}, Proposition~4.2)} Suppose $L$ is unimodular and $H\triangleleft L$ is a closed normal subgroup. If $H$ and $L/H$ are unimodular and satisfy $\BM(r)$ and $\BM(s)$ respectively, then $L$ satisfies $\BM(r+s)$.
 \item {\rm (Modular extension, \cite{JTZ}, Proposition~5.2)} Suppose $H=\ker\Delta_L$ satisfies $\BM(r)$ and $\Delta_L:L\to\R^{>0}$ is a quotient map.  Then $L$ satisfies $\BM(r+1)$.
 \item {\rm (Cocompact-factor principle, \cite{JTZ}, Proposition~6.2)} Suppose $L$ is connected and unimodular, and $L=KH$, where $H$ is connected and closed, $K$ is connected and unimodular, and $K\cap H$ is compact.  If $H$ satisfies $\BM(r)$, then $L$ satisfies $\BM(r)$.
 \item {\rm (Reduction theorem)} Let $L$ be a Lie group of dimension $d$. If all the simply connected simple Lie groups $G$ with dimension at most $d$ satisfy $\BM(\ndim(G))$, then $L$ satisfies $\BM(\ndim(L))$. 
\end{enumerate}
\end{fact}

We remark that Fact~\ref{fact:jtz}(5) is Theorem~1.5 in \cite{JTZ}. Although the statement in \cite{JTZ} only claims the sharp nonabelian Brunn--Minkowski inequality follows from the corresponding result for simply connected simple groups, but the proof of Theorem~1.5 actually gives us the stronger version with dimension $d$ stated here. 

The proof of Fact~\ref{fact:jtz}(4) uses the following product integration formula, which we also need directly.

\begin{fact}[Proposition 5.26  \cite{Knapp}]
\label{fact:product-integration}
Let $G$ be a connected unimodular Lie group and suppose $G=ST$, where $S,T$ are closed and $S\cap T$ is compact.  Haar measures may be normalized so that
\begin{equation}\label{eq:product-integration}
  \int_G f(g)\d\mu_G(g)
  =\int_S\int_T f(st)\d\nu_T(t)\d\mu_S(s)
\end{equation}
for every $f\in C_c(G)$.
\end{fact}

We also use the standard fact that $\ker\Delta_L$ is unimodular; see
\cite[Section 4]{Folland}.

\subsection{Noncompact Lie dimension and helix dimension}

We next recall the two invariants of~\cite{JTZ} introduced in Section~\ref{sec:introduction}.  The noncompact Lie dimension supplies the exponent in our main theorem, while the helix dimension organizes the case division below.

For a connected Lie group $G$, the noncompact Lie dimension $\ndim(G)$ is given by \eqref{eq:noncompact-dim}. Note that it is additive in every exact sequence of \emph{connected} Lie groups; see \cite[Proposition~2.12(1)]{JTZ} and the example after it. We will always use this together with Fact~\ref{fact:jtz}(2).

The helix dimension $\hdim(G)$ of a connected simple Lie group $G$ is the rank of its discrete center, and it is at most one \cite[Section~2]{JTZ}.  Recall that a simply connected simple real Lie group has infinite center exactly when its symmetric space is Hermitian, or equivalently, when a maximal compactly embedded Lie algebra has a one-dimensional center; see \cite[Chapter~VIII]{Helgason} and \cite[Chapter~VII]{Knapp}. The Hermitian list is
\begin{equation}\label{eq:hermitian-list}
\begin{split}
 \mathfrak{su}(p,q),\quad
 \mathfrak{sp}_{2n}(\R),\quad
 \mathfrak{so}^*(2n),\quad
 \mathfrak{so}(2,n),\quad
 \mathfrak e_{6(-14)},\quad
 \mathfrak e_{7(-25)}.
\end{split}
\end{equation}
We use the usual low rank identifications, including $\mathfrak{sp}_2(\R)\simeq\mathfrak{sl}_2(\R)$ and $\mathfrak{so}^*(6)\simeq\mathfrak{su}(1,3)$.

By the helix-free theorem of~\cite{JTZ}, every connected simple Lie group $G$ with finite center satisfies $\BM(\ndim(G))$. Hence, it remains to treat the simply connected groups whose Lie algebras appear in \eqref{eq:hermitian-list}.

\subsection{The Iwasawa decomposition} The Iwasawa decomposition is a useful tool used in the proof of Fact~\ref{fact:jtz}(4) in \cite{JTZ}. In this subsection we collect some properties of the Iwasawa coordinates that we need. 

Let $G$ be a semisimple Lie group, we fix an Iwasawa decomposition $G=KAN$ and set $H=AN$. Recall from~\cite[Theorems~6.31 and~6.46]{Knapp} that $K$ is the analytic subgroup attached to a maximal compactly embedded subalgebra $\mathfrak k$ of $\mathfrak g$, that $K$, $A$, and $N$ are closed and connected, that $Z(G)\subseteq K$, that $A$ and $N$ are simply connected and are diffeomorphic to Euclidean spaces, and that the multiplication map $K\times A\times N\to G$ is a diffeomorphism. 

\begin{fact} \label{fact: iwasawa}
Let $G=\widetilde{\PSL}_2(\R)$, let $K$ and $H$ be as above, we have the following.
\begin{enumerate}[(1)]
\item The group $K$ is isomorphic as a Lie group to $(\R,+)$.  We fix
  such an isomorphism once and for all and write
  \[
    K=\{k_t:t\in\R\},\qquad k_tk_s=k_{t+s}.
  \]
\item The two multiplication maps
  \begin{equation}\label{eq:two-global-coordinates}
    K\times H\to G,\quad (k,h)\mapsto kh,
    \qquad
    H\times K\to G,\quad (h,k)\mapsto hk
  \end{equation}
  are diffeomorphisms.  In particular, $G=KH=HK$ and $K\cap H=\{e\}$.
\item Every compact subgroup of $G$ is trivial.  Consequently,
  $\ndim(G)=\dim(G)=3$.
\item A right Haar measure $\nu_H$ and a left Haar measure $\mu_H$ on
  $H$ can be normalized so that
  \begin{align}
   \int_G f(g)\d\mu_G(g)
   &=\int_{\R}\int_H f(k_th)\d\nu_H(h)\d t,
   \label{eq:KH-haar}\\
   &=\int_H\int_{\R}f(hk_s)\d s\d\mu_H(h)
   \label{eq:HK-haar}
  \end{align}
  for every $f\in C_c(G)$, where $\d t$ and $\d s$ denote the Lebesgue measure on $\R$.  These two identities determine $\nu_H$ and $\mu_H$ uniquely.
\item The subgroup $H$ is a $\mu_G$-null subset of $G$.
\end{enumerate}
\end{fact}
The above facts can be found in \cite[Theorem 6.46, Corollary 8.31, Theorem 8.32]{Knapp}, and corollaries of those using basic properties of $\widetilde{\PSL}_2(\R)$. 

We will also use following facts from the structure theory of Lie groups. Fact~\ref{fact:standard-structure}(1) is \cite[Theorems~12.1.17 and~12.1.18]{HilgertNeeb}, and Fact~\ref{fact:standard-structure}(2) can be derived from \cite[Theorems~14.1.3 and~14.3.11]{HilgertNeeb}.

\begin{fact}\label{fact:standard-structure}
Let $L$ be a connected Lie group. 
\begin{enumerate}[(1)]
\item Let the Lie algebra  $\mathfrak l$ of $L$ is compact. Then $L$ is isomorphic to a direct product
\[
   L\simeq M\times V,
\]
where $M$ is the unique maximal compact subgroup of $L$ and $V$
is a vector group. In particular, suppose that $L$ is simply connected and
\[
   \mathfrak l=\mathfrak c\oplus\mathfrak z,
\]
where $\mathfrak c$ is compact semisimple and $\mathfrak z$ is
central. If $C$ is the analytic subgroup with Lie algebra
$\mathfrak c$, then $C$ is compact and connected and
\[
   L\simeq C\times\mathbb R^{\dim\mathfrak z}.
\]
\item Let $D\leq L$ be compact. Then $D$ is contained in a maximal compact subgroup $M$ of $L$. Furthermore,
\[
   L/M\simeq\mathbb R^{\dim L-\dim M}
\]
as smooth manifolds. 
Consequently, if $L/D$ is contractible, then $D$ is already a maximal compact subgroup of $L$.
\end{enumerate}
\end{fact}

\section{One dimensional rearrangement and transportation}

We record the precise nonsmooth form of one dimensional monotone transport that will be used in later sections. The ambient group is $\R$ and all measures in this section are Lebesgue measures, which is denoted by $\lambda$. For a transportation map $T$ we use $T_\#$ for the pushforward by $T$. The proof of the lemma is mostly standard and we record a proof here for the sake of completeness. 

\begin{lemma}
\label{lem: monotone rearrangement}
Let $f,g:\mathbb R\to[0,\infty)$ be bounded, compactly supported Borel functions, and suppose that $\alpha,\beta>0$ with
\[
  \alpha=\int_{\mathbb R}f(t)\d t,
  \qquad\text{and}\qquad
  \beta=\int_{\mathbb R}g(s)\d s.
\]
Then there exist a nondecreasing Borel map $T:\mathbb R\to\mathbb R$ and a Borel set $\Omega\subseteq \mathrm{supp} (f)$ such that
\[
  \int_{\mathbb R\setminus\Omega}f(t)\d t=0,
\]
and the following properties hold.
\begin{enumerate}[(1)]
\item The map $T$ transports the normalized measure with density $f$ to the normalized measure with density $g$
\[
  T_{\#}\left(\frac{f(t)}{\alpha}\d t\right)
  =\frac{g(s)}{\beta}\d s.
\]

\item For every $t\in\Omega$, the map $T$ is differentiable at $t$, with $0<T'(t)<\infty$, and
\begin{equation}
  T'(t)g(T(t))=\frac{\beta}{\alpha}f(t).
  \label{eq:one-dimensional-jacobian}
\end{equation}

\item There are pairwise disjoint $\sigma$-compact sets $S_j\subseteq\Omega$ for $j\geq1$, such that $T|_{S_j}$ is Lipschitz for every $j$ and
\[
  \lambda\left(\Omega\setminus  \bigcup_{j=1}^{\infty}S_j\right)=0.
\]
\end{enumerate}
\end{lemma}

\begin{proof} We first prove (1). The proof follows the same argument~\cite[Proposition 2.2, Theorem 2.5]{optimaltrans}. 
We first construct the increasing transport. Introduce the probability measures
\[
  \d\rho(t)=\frac{f(t)}{\alpha}\d t,
  \qquad\text{and}\qquad
  \d\sigma(s)=\frac{g(s)}{\beta}\d s,
\]
and their distribution functions
\[
  F(t) =\frac{1}{\alpha}\int_{-\infty}^{t}f(u)\d u,
  \qquad
  G(s)=\frac{1}{\beta}\int_{-\infty}^{s}g(v)\d v.
\]
Note that both $F$ and $G$ are continuous and nondecreasing, and both have limits $0$ and $1$ at $-\infty$ and $+\infty$, respectively.

Choose $a<b$ such that $\mathrm{supp}(g)\subseteq[a,b]$. Define $Q:[0,1]\to[a,b]$ by $Q(0)=a$, $Q(1)=b$, and for $0<u<1$,
\[
  Q(u)=\inf\{s:G(s)\geq u\}
\]
Define $T=Q\circ F$. The maps $Q$ and $T$ are nondecreasing and Borel, and $T$ is bounded.

We verify the pushforward identity directly. Since $\rho$ is atomless, its distribution function $F$ is continuous. Hence $F(t)$ is uniformly distributed on $(0,1)$ when $t$ is distributed according to $\rho$, that is $F_\#\rho=\lambda|_{(0,1)}$. 
On the other hand, by the definition of $Q$, for $0<u<1$, $Q(u)\leq s$ is equivalent to $u\leq G(s)$. 
It follows that
\[
  \lambda(\{u\in(0,1):Q(u)\leq s\})=G(s),
\]
so $Q_{\#}(\lambda |_{(0,1)})=\sigma$. Consequently,
$
  T_{\#}\rho=Q_{\#}(F_{\#}\rho)=\sigma,
$
which proves (1).

\smallskip
Let us now prove (2). 
Let $L_f$ and $L_g$ be the sets of Lebesgue points of $f$ and $g$,
respectively, and let
\[
  D_T=\{t\in\mathbb R:T'(t)\text{ exists and is finite}\}.
\]
These sets may be chosen Borel. The Lebesgue differentiation theorem and the almost-everywhere differentiability of monotone functions give
$
  \lambda(\mathbb R\setminus L_f)
  =\lambda(\mathbb R\setminus L_g)
  =\lambda(\mathbb R\setminus D_T)=0.
$
We now set $Y=L_g\cap\{g>0\}$ and $I=\{t\in\mathbb R:0<F(t)<1\}$,
and define
\[
  \Omega=\mathrm{supp}(f)\cap L_f\cap D_T\cap I\cap T^{-1}(Y).
\]
Then $\Omega$ is a Borel subset of $\mathrm{supp}(f)$. Moreover, $\sigma(\mathbb R\setminus Y)=0$, so (1) implies
$\rho\bigl(T^{-1}(\mathbb R\setminus Y)\bigr)=0$.
We also have $\rho(\mathbb R\setminus I)=0$, simply because $F_{\#}\rho$ is the Lebesgue measure on $(0,1)$. The other sets removed in the definition of $\Omega$ are $\rho$-null by absolute continuity of $\rho$. Thus $\rho(\mathbb R\setminus\Omega)=0$.

For every $0<u<1$, continuity of $G$ gives
$
  G(Q(u))=u.
$
Indeed, the definition of $Q(u)$ gives $G(Q(u))\geq u$, while $G(s)<u$ for $s<Q(u)$. Letting $s\to Q(u)$ gives the claimed inequality. As a Consequence, for $0<F(r)<1$ we have
$ G(T(r))=F(r).$ If $t\in\Omega$, then $F(t)\in(0,1)$. Since $F$ is continuous, the preceding identity holds in a neighborhood of $t$. At this point,
\[
  F'(t)=\frac{f(t)}{\alpha},
  \qquad\text{and}\qquad
  G'(T(t))=\frac{g(T(t))}{\beta},
\]
and $T'(t)$ exists and is finite. The ordinary chain rule applied to $G\circ T=F$ therefore gives
\[
  \frac{g(T(t))}{\beta}T'(t)=\frac{f(t)}{\alpha}.
\]
This is \eqref{eq:one-dimensional-jacobian}. Since $f(t)>0$ and $g(T(t))>0$ on $\Omega$, the same identity yields $0<T'(t)<\infty$.

Finally we prove (3). For integers $m,n\geq1$, let
\[
  E_{m,n}:=\left\{t\in\Omega:
  |T(q)-T(t)|\leq m|q-t|
  \text{ for every }q\in\mathbb Q
  \text{ with }0<|q-t|<\frac{1}{n}\right\}.
\]
These sets are Borel as the defining conditions are countable. Since $T'(t)$ is finite for every $t\in\Omega$, differentiability gives
\[
  \Omega=\bigcup_{m,n\geq1}E_{m,n}.
\] 
Let
$
  M=\sup_{\mathbb R}T-\inf_{\mathbb R}T<\infty.
$
Suppose that $t,u\in E_{m,n}$ and $0<|t-u|<1/n$. Choose rational numbers $q_k\to u$ with $0<|q_k-t|<1/n$. Since $u\in\Omega$, the map $T$ is differentiable at $u$, hence continuous at $u$. Therefore $T(q_k)\to T(u)$. Letting $k\to\infty$ in the defining estimate gives
\[
  |T(t)-T(u)|\leq m|t-u|.
\]
If $|t-u|\geq1/n$, boundedness of $T$ gives instead
\[
  |T(t)-T(u)|\leq M\leq nM|t-u|.
\]
Hence $T|_{E_{m,n}}$ is Lipschitz with Lipschitz constant at most
$\max\{m,nM\}$. Since $\Omega\subseteq\mathrm{supp}(f)$ which is bounded, every $E_{m,n}$ has finite Lebesgue measure. By inner regularity, each $E_{m,n}$ can be covered by countably many compact
subsets of itself up to a set of measure $0$. Enumerating these compact subsets gives the family $\{P_i\}_{i=1}^\infty$. Finally we disjointify the compact Lipschitz pieces by setting $S_1=P_1$, and for $j\geq2$
\[
  S_j=P_j\setminus\bigcup_{i<j}P_i.
\]
The sets $S_j$ are pairwise disjoint and $\sigma$-compact as $P_j$ are compact. They cover $\Omega$ up to a set of measure $0$, and $T|_{S_j}$ is still Lipschitz. This proves (3). 
\end{proof}

\section{Ordered Iwasawa coordinates for $\widetilde{\PSL}_2(\R)$}\label{sec:ordered}

In this section and the next two, we prove the base case of our argument, that is the sharp inequality $\BM(3)$ for the universal cover of $\PSL_2(\R)$. The precise statement is Theorem~\ref{thm: PSL2}. Throughout Sections 4--6, we fix
\[
  G=\widetilde{\PSL}_2(\R),
\]
As in Fact~\ref{fact: iwasawa}, we fix an Iwasawa decomposition $G=KAN$ of $G$, and we set $H=AN$. We also fix the covering map $\pi:G\to\PSL_2(\R)$, with kernel the center $Z(G)$ of $G$ and it is infinite cyclic. 

Note that $G$ is simple and hence unimodular. To avoid confusion, we will therefore not use $\nu_G$ below, and $\mu_G$ will always denote a fixed Haar measure on $G$, which is both left and right invariant. On the other hand, the group $H$ is solvable and not unimodular, so on $H$ we will still use a left Haar measure $\mu_H$ and a right Haar
measure $\nu_H$ separately

In light of Fact~\ref{fact: iwasawa}(2), every $g$ in $G$ can
be written uniquely as $g=k_th$ with $t\in\R$ and $h\in H$.  From this we define a smooth coordinate
\begin{equation}\label{eq:q-coordinate}
  q:G\to\R,\qquad q(k_th)=t,
\end{equation}
which records the $K$-coordinate of $g$ in the $KH$ order.  The level sets of $q$ are the left cosets of $H$.  In particular $q^{-1}(0)=H$, which has $\mu_G$-measure $0$.

The other order $HK$ tells us how to move a copy of $K$ past a given element of $H$.  More precisely, for $h\in H$ and $s\in\R$, we view $hk_s$ as an element of $G=HK$, and then use the order $KH$ to write it uniquely as
\begin{equation}\label{eq:phi-def}
  hk_s=k_{\phi_h(s)}h_s,
\end{equation}
where $h_s\in H$. 
In other words, $\phi_h(s)=q(hk_s)$. The next lemma records some properties of the map $\phi_h:\R\to\R$. Here we say two elements $s,t\in\R$ have the same sign if $t>0$ when $s>0$,  $t=0$ when $s=0$, and  $t<0$ when $s<0$. 

\begin{lemma}\label{lem:order}
Suppose $h\in H$.  Then the map $\phi_h:\R\to\R$ defined by
\eqref{eq:phi-def} is an increasing homeomorphism with $\phi_h(0)=0$. Therefore for all $t,s\in\R$ we have
\begin{equation}\label{eq:sign-formula}
 q(k_thk_s)=t+\phi_h(s),
\end{equation}
and $\phi_h(s)$ has the same sign as $s$.
\end{lemma}

\begin{proof}
We first identify $\phi_h$ as the map induced by a left translation. By Fact~\ref{fact: iwasawa}(2), every left coset of $H$ contains exactly one element of $K$, and the map
\[
  \psi:\R\to G/H,\qquad t\mapsto k_tH,
\]
is a continuous bijection. Note that $q$ is constant on left cosets of $H$ and hence it induces a continuous map $\bar q:G/H\to\R$, which is the inverse of $\psi$. Therefore $\psi$ is a homeomorphism, and we use it to identify $G/H$ with $\R$ for the rest of the proof.  

Under this identification, the action of $h$ on $G/H$ by left translation is exactly $\phi_h$. Indeed, by \eqref{eq:phi-def},
\[
  h(k_sH)=k_{\phi_h(s)}h_sH=k_{\phi_h(s)}H.
\]
Left translation by $h$ is a homeomorphism of $G/H$.  Hence $\phi_h$ is a homeomorphism of $\R$, and it fixes the origin since $hH=H$. So we have $\phi_h(0)=0$ as desired.

We next show that $\phi_h$ is increasing. Note that a homeomorphism of $\R$ is either increasing or decreasing, so we may define a map
\[
\theta:H\to\{\pm1\}
\]
by defining $\theta(h)=1$ when $\phi_h$ is increasing and $\theta(h)=-1$ otherwise. This is the orientation character of the action of $H$ on $G/H$. As $h\mapsto\phi_h$ is a left action and hence $\phi_{hh'}=\phi_h\circ\phi_{h'}$. Now $\phi_h$ is injective and $\phi_h(0)=0$, so $\phi_h(1)\neq0$, and $\theta(h)$ is the sign of $\phi_h(1)$. The map $H\to\R$, $h\mapsto\phi_h(1)=q(hk_1)$ is continuous by \eqref{eq:q-coordinate} and Fact~\ref{fact: iwasawa}(2). Therefore 
$\theta$ is continuous. Finally, the group $H=AN$ is connected and $\theta(e)=1$, so $\theta$ is constantly $1$. In other words, $\phi_h$ is increasing for every $h$ in $H$.

It remains to deduce \eqref{eq:sign-formula} and the sign property. Using \eqref{eq:phi-def} and the fact $k_tk_{t'}=k_{t+t'}$ in $K$, we get
\[
  k_thk_s=k_t\,k_{\phi_h(s)}h_s=k_{t+\phi_h(s)}h_s,
\]
so \eqref{eq:sign-formula} follows from the definition
\eqref{eq:q-coordinate} of $q$. For the sign rule, the map $\phi_h$ is increasing with $\phi_h(0)=0$, so $\phi_h(s)>0$ for $s>0$ and $\phi_h(s)<0$ for $s<0$, which is the desired conclusion.
\end{proof}

It is worth explaining that passing to the universal cover
contributes here.  Lemma~\ref{lem:order} says that $H$ acts on the one-dimensional homogeneous space $G/H$ by increasing homeomorphisms fixing a point. In the next two sections we use the dichotomy that comes with this order: the complement of the fixed point in $G/H\cong\R$ has two components, and \eqref{eq:sign-formula} records which of the two a given product falls into. For $\PSL_2(\R)$ itself the corresponding homogeneous space is the flag variety $\PSL_2(\R)/AN$, which is a circle, and $AN$ still acts on it by orientation-preserving homeomorphisms fixing the base point. However, the complement of a point in a circle is connected, so no dichotomy of the above kind is available there.  Passing to the universal cover unwinds this circle into the line $G/H\cong\R$.

We have the following corollary on Lie algebras. We write $\mathfrak k, \mathfrak h, \mathfrak g$ for the Lie algebras of $K,H,G$ respectively as usual, and $\mathrm{Ad}_hZ$ the adjoint action of $h$ to the Lie algebra vector $Z$. 
\begin{corollary} \label{cor:quotient density}
Let $Z\in\mathfrak k$ be determined by $k_t=\exp(tZ)$. There is a smooth positive character
\[
\vartheta:H\to\mathbb R_{>0}
\]
such that for every $h\in H$ there is a unique vector $Y_h\in\mathfrak h$ satisfying
\begin{equation}
  \mathrm{Ad}_hZ=\vartheta(h)Z+Y_h.
  \label{eq:adjoint-decomposition}
\end{equation}
Moreover, we have $
  \vartheta(h)=\phi_h'(0)=\Delta_H(h)$, and hence $
  \d\mu_H(h)=\vartheta(h)\d\nu_H(h)$.
\end{corollary}

\begin{proof}
Differentiating the diffeomorphism $K\times H\to G$ at the identity gives the direct-sum decomposition $\mathfrak g=\mathbb R Z\oplus\mathfrak h$. Write $\varpi:\mathfrak g \to \mathfrak g/\mathfrak h$ the projection map, thus $\mathfrak g/\mathfrak h$ is one-dimensional and $\varpi(Z)$ is a basis of this quotient.

For $h\in H$, the map $\mathrm{Ad}_h$ preserves
$\mathfrak h$. It induces a linear automorphism of $\mathfrak g/\mathfrak h$. Since this quotient is one-dimensional, there
is a unique nonzero scalar $\vartheta(h)$ such that
\[
  \varpi\bigl(\mathrm{Ad}_hZ\bigr) =\vartheta(h)\varpi(Z).
\]
The induced adjoint action is a smooth representation of $H$ on a one-dimensional vector space. As a consequence, $\vartheta$ is a smooth character. The direct-sum decomposition then gives the unique vector
\[
  Y_h=\mathrm{Ad}_hZ-\vartheta(h)Z\in\mathfrak h.
\]

We next identify the scalar $\vartheta(h)$. Recall the standard tangent space identification
\[
  T_{eH}(G/H)\simeq\mathfrak g/\mathfrak h.
\]
For every $s$, observe that $h\exp(sZ)H
  =\exp\bigl(s\mathrm{Ad}_hZ\bigr)H$.
On the other hand, the definition of $\phi_h$ in Lemma~\ref{lem:order} gives
\[
  h\exp(sZ)H=k_{\phi_h(s)}H
  =\exp\bigl(\phi_h(s)Z\bigr)H.
\]
Differentiating the last two expressions at $s=0$ in $T_{eH}(G/H)\simeq\mathfrak g/\mathfrak h$ yields
\[
  \varpi\bigl(\mathrm{Ad}_hZ\bigr)
  =\phi_h'(0)\varpi(Z).
\]
Since $\varpi(Z)\neq0$, we obtain $\vartheta(h)=\phi_h'(0)>0$ which shows that the character $\vartheta$ takes values in $\mathbb R_{>0}$.

It remains to relate this coefficient to the modular function.
The short exact sequence
\[
  0\to\mathfrak h\to\mathfrak g\to\mathfrak g/\mathfrak h
  \to0
\]
is invariant under $\mathrm{Ad}_h$. Therefore the absolute determinant factors as
\[
  \left|\det\bigl(\mathrm{Ad}_h|_{\mathfrak g}\bigr)\right|
  =\left|\det\bigl(\mathrm{Ad}_h|_{\mathfrak h}\bigr)\right|
  \vartheta(h).
\]
Since the group $G$ is unimodular, so the left hand side is $1$. As
$
  \Delta_H(h)
  =\left|\det(\mathrm{Ad}_h|_{\mathfrak h})\right|^{-1},
$
we conclude $\Delta_H(h)=\vartheta(h)$.
\end{proof}

\section{The transported Iwasawa diagonal product}
\label{sec:transported-iwasawa-diagonal}

We now combine the one-dimensional transport with the two orders $KH$ and $HK$ of the Iwasawa decomposition. For product set $XY$, the first order produces right Haar fibres of $X$, and the second produces left Haar fibres of $Y$. We couple the two fibre-mass distributions by the rearrangement lemma and map each moving product fibre into $XY$. Lemma~\ref{lem:order} provides global injectivity, while Corollary~\ref{cor:quotient density} provides the modular term in the Jacobian. 

Let $X,Y\subset G$ be compact sets of positive Haar measure. In the two orders of the Iwasawa decomposition, define the fiber
\[
  C_t=\{h\in H:k_t h\in X\},
  \qquad\text{and}\qquad
  D_s=\{h\in H:h k_s\in Y\}.
\]
Every $C_t$ and $D_s$ is compact. Define the corresponding fibre measure functions by
\[
  f(t)=\nu_H(C_t),
  \qquad \text{and}\qquad
  g(s)=\mu_H(D_s).
\]

Clearly $f$ and $g$ have compact support and are bounded, and they are both Borel measurable. 
The two integration formulas in Fact~\ref{fact: iwasawa} gives $\mu_G(X) = \int_{\mathbb R}f(t)\d t$ and $\mu_G(Y)=\int_{\mathbb R}g(s)\d s$. Write $\alpha = \mu_G(X)$ and $\beta=\mu_G(Y)$. 

Apply Lemma~\ref{lem: monotone rearrangement} to $f$ and $g$. Let $T:\mathbb R\to\mathbb R$ and $\Omega\subset\mathrm{supp}(f)$ be the resulting transport map and full measure set, and for $t\in\Omega$ write $p(t)=T'(t)$. 
Thus $T$ is nondecreasing, $0<p(t)<\infty$ on $\Omega$, and
$p(t)g(T(t))=\frac{\beta}{\alpha}f(t)$.
For $t\in\Omega$, write \[
 E_t=C_tD_{T(t)}.
\]

\begin{proposition}
\label{prop: transported diagonal}
With the notation defined above,
\begin{equation}
  \mu_G(XY)\geq
  \int_{\Omega}\left(\nu_H(E_t)+p(t)\mu_H(E_t)\right)\d t.
  \label{eq:transported-diagonal-bound}
\end{equation}
\end{proposition}

\begin{proof}
The proof has three parts. We first establish the set theoretic inclusion and injectivity, then compute the Jacobian at differentiability points of $T$, and finally justify the area formula on measurable Lipschitz pieces.

\smallskip
Let us first develop the transported diagonal and its injectivity.
Consider the moving domain
\[
  \mathcal D=\{(t,h)\in\Omega\times H:h\in E_t\},
\]
and for $(t,h)\in\mathcal D$ define
\[
  \Psi(t,h)=k_t h k_{T(t)}.
\]
If $h\in E_t$, choose $c\in C_t$ and $d\in D_{T(t)}$ such that $h=cd$. Then
\[
  \Psi(t,h)=k_tcdk_{T(t)}=(k_t c)(d k_{T(t)})\in XY.
\]
Therefore $\Psi(\mathcal D)\subseteq XY$.

We next prove that $\Psi$ is injective. Suppose it is not and that $\Psi(t,h)=\Psi(t',h')$ for two points in $\mathcal D$. Note that if $t=t'$, then $T(t)=T(t')$, and cancellation on the left and right gives $h=h'$.

Suppose instead that $t<t'$. Since $T$ is nondecreasing, $T(t')-T(t)\geq0$, and 
\[
  h=k_{t'-t}\,h'\,k_{T(t')-T(t)}.
\]
As $h\in H$, so the $q$-coordinate of the right hand side is zero. On the other hand, the order identity (Lemma~\ref{lem:order}) gives us
\[
  q\bigl(k_{t'-t}h'k_{T(t')-T(t)}\bigr)
  =(t'-t)+\phi_{h'}\bigl(T(t')-T(t)\bigr)>0
\]
as the first term is strictly positive and the second is nonnegative, it is a contradiction. Hence $\Psi$ is injective on all of $\mathcal D$.

\smallskip
Let us now compute the Jacobian. We compute how $\Psi$ changes in its three input directions: one direction comes from the real variable $t$, and two directions come from the two dimensional group $H$.

Recall that $\mathfrak h=T_eH$ is the tangent space of $H$ at its identity.  Choose a basis $X_1,X_2$ of $T_eH$ normalized to have unit density with respect to $\nu_H$ at $e$. Since right multiplication $R_h:x\mapsto xh$ preserves the right Haar measure $\nu_H$, its differential preserves the corresponding density. Therefore for $i=1,2$
\[
X_i^R(h)=\d(R_h)_eX_i,
\]
form a unit density basis at $h$. Since $\partial_t$ has unit density for $\d t$, the ordered basis
\[
\bigl(\partial_t,X_1^R(h),X_2^R(h)\bigr)
\]
has unit density for the product measure $\d t\d\nu_H$.

We next choose volume coordinates on the target. Let $Z$ be the tangent vector determined by $k_t=\exp(tZ)$. The derivative at $(0,e)$ of the coordinate map $(t,h)\mapsto k_th$ sends
$(\partial_t,X_1,X_2)$ to $(Z,X_1,X_2)$. By Fact~\ref{fact: iwasawa}(4) the latter ordered basis has $\mu_G$-volume $1$ at the identity of $G$.

The derivative of $\Psi$ at $(t,h)$ produces tangent vectors at the point $g=\Psi(t,h)$. In order to express all of them in the fixed basis $(Z,X_1,X_2)$, we move them back to the identity by a right translation $g^{-1}$. Therefore, for $v\in T_gG$, write
\[
  \omega_g^R(v)=\d(R_{g^{-1}})_g v,
\]
which is known as the right Maurer form.  Since $\mu_G$ is right invariant, moving the three output vectors in this way does not change their volume.

Fix $t\in\Omega$ and write $p(t)=T'(t)$. Recall that $\mathrm{Ad}_aX$ is the tangent vector obtained by differentiating $a\exp(rX)a^{-1}$ at $r=0$. Direct differentiation of $\Psi(t,h)$ for $i=1,2$ then gives
\begin{align}
  \omega_{\Psi(t,h)}^R
    \bigl(\d\Psi_{(t,h)}(\partial_t)\bigr)
  &=\mathrm{Ad}_{k_t}
    \bigl(Z+p(t)\mathrm{Ad}_hZ\bigr),
  \label{eq:transported-diagonal-t-derivative}\\
  \omega_{\Psi(t,h)}^R
    \bigl(\d\Psi_{(t,h)}(0,X_i^R(h))\bigr)
  &=\mathrm{Ad}_{k_t}X_i.
  \label{eq:transported-diagonal-h-derivative}
\end{align}
Indeed, when $t$ changes, the leftmost factor $k_t$ contributes the direction $Z$, while the rightmost factor $k_{T(t)}$ contributes $p(t)\mathrm{Ad}_{k_th}Z$, this gives \eqref{eq:transported-diagonal-t-derivative}. 
When $h$ moves in the direction $X_i^R(h)$, we use the curve $r\mapsto\exp(rX_i)h$. After moving the resulting output vector back to the identity, it becomes $\mathrm{Ad}_{k_t}X_i$ which gives \eqref{eq:transported-diagonal-h-derivative}.

Since $G$ is unimodular, we have
\[
  \left|\det\bigl( \mathrm{Ad}_{k_t}|_{\mathfrak g}\bigr)\right|=1.
\]
Thus applying $\mathrm{Ad}_{k_t}$ to all three columns does not change the absolute value of their determinant. Since we can remove this common map, the three columns are now
\[
  Z+p(t)\mathrm{Ad}_hZ,\qquad X_1,\qquad X_2.
\]
By Corollary~\ref{cor:quotient density}, we have 
\[
  \mathrm{Ad}_hZ=\vartheta(h)Z+Y_h. 
\]
We now write $Y_h=a_1X_1+a_2X_2$. Relative to the ordered basis $(Z,X_1,X_2)$, the three columns form the matrix
\[
  \begin{pmatrix}
    1+p(t)\vartheta(h) & 0 & 0\\
    p(t)a_1            & 1 & 0\\
    p(t)a_2            & 0 & 1
  \end{pmatrix}.
\]
As the matrix is triangular, its determinant is  $1+p(t)\vartheta(h)$. It is positive because $p(t)>0$ and $\vartheta(h)>0$. Hence the Jacobian relative to $\d t\d\nu_H$ and $\mu_G$ is
\begin{equation}
  J_{\Psi}(t,h)=1+p(t)\vartheta(h).
  \label{eq:transported-diagonal-jacobian}
\end{equation}

Finally, let us justify the area formula. By Lemma~\ref{lem: monotone rearrangement}, there are pairwise disjoint Borel sets
$S_j\subseteq \Omega$ such that $T|_{S_j}$ is Lipschitz and the union of $S_j$ is $\lambda$-almost $\Omega$.
For each $j$, define the triple incidence set
\[
  R_j=\left\{(t,c,d)\in S_j\times H\times H:
    c\in C_t, d\in D_{T(t)}\right\}.
\]
Note that $ R_j$ is $\sigma$-compact. Its continuous image
\[
  Q_j=\{(t,cd):(t,c,d)\in R_j\}\subseteq S_j\times H
\]
is also $\sigma$-compact and hence Borel, and its vertical section over $t\in S_j$
is exactly $(Q_j)_t=C_tD_{T(t)}=E_t.$

As $T$ only Lipschitz on $S_j$, we will need to extend it to all $\R$ in order to apply the area formula. 
Let $L_j$ be a Lipschitz constant for $T|_{S_j}$. Define
\[
  T_j(x)=\inf_{y\in S_j}\{T(y)+L_j|x-y|\}. 
\]
Then $T_j$ is Lipschitz on $\mathbb R$ and agrees with $T$ on $S_j$. We may similarly define
\[
  \Psi_j(t,h)=k_t h k_{T_j(t)}.
\]
This map is locally Lipschitz on $\mathbb R\times H$, and it agrees with $\Psi$ on $Q_j$. In particular, $\Psi_j|_{Q_j}$ is injective.

Next, we claim that for almost every $t\in S_j$, $T_j'(t)=T'(t)=p(t).$
Indeed, almost every point $t\in S_j$ is both a density point of $S_j$ and a differentiability point of $ T_j$. At such a point, choose distinct $t_n\in S_j$ with $t_n\to t$. Since $T_j=T$ on $S_j$, 
\[
\frac{T_j(t_n)- T_j(t)}{t_n-t}=\frac{T(t_n)-T(t)}{t_n-t}.
\]
As $T$ is differentiable at every point of $\Omega$, and therefore they are equal.

As each $Q_j$ is a Borel subset of the Polish manifold $\mathbb R\times H$, and $\Psi_j|_{Q_j}$ is continuous and injective. By Lusin's theorem $\Psi_j(Q_j)$ is also Borel in $G$. Note also that these images are pairwise disjoint. Indeed, $\Psi$ is injective, hence an equality between a point of $\Psi_i(Q_i)$ and a point of $\Psi_j(Q_j)$ would imply the equality of the two $t$-coordinates, which is impossible for $i\neq j$ since $S_i\cap S_j=\varnothing$ from Lemma~\ref{lem: monotone rearrangement}.

We can now apply the area  formula \cite[Section~3.3]{EG92} to $\Psi_j$ on $Q_j$. More precisely, $Q_j$ lies in a compact subset of $\mathbb R\times H$, hence it can be divided into finitely many Borel covers where the usual Euclidean area formula applies. The exceptional values of $t$ with $T_j'(t)\neq p(t)$ contribute zero product measure. Using injectivity, the Jacobian \eqref{eq:transported-diagonal-jacobian}, and the identity $\d\mu_H=\vartheta\d\nu_H$ from Corollary~\ref{cor:quotient density}, we obtain
\[
  \begin{aligned}
  \mu_G\bigl(\Psi_j(Q_j)\bigr)
  &=\int_{Q_j}
    \bigl(1+p(t)\vartheta(h)\bigr)\d t\d\nu_H(h)=\int_{S_j}\int_{E_t}
    \bigl(1+p(t)\vartheta(h)\bigr)\d\nu_H(h)\d t\\
  &=\int_{S_j}\bigl(\nu_H(E_t)+p(t)\mu_H(E_t)\bigr)\d t.
  \end{aligned}
\]
Finally, every $\Psi_j(Q_j)$ is contained in $XY$, and these sets are pairwise disjoint. Therefore
\[
  \mu_G(XY)\geq\sum_{j=1}^{\infty}\mu_G\bigl(\Psi_j(Q_j)\bigr)=\int_{\Omega}
    \bigl(\nu_H(E_t)+p(t)\mu_H(E_t)\bigr)\d t.
\]
as desired. 
\end{proof}

\section{Proof of the base case: $\widetilde{\PSL}_2(\R)$}

In this section, we prove the sharp Brunn--Minkowski inequality for the group $\widetilde{\PSL}_2(\R)$. Recall that $H$ is isomorphic to the $ax+b$ group, and we make use of the Brunn--Minkowski theorem for $H$ which is proven in~\cite{JTZ}.
\begin{fact}\label{fact: H}
    Let $X,Y\subseteq H$ be compact sets of positive measure. Then
\[
\left(\frac{\nu_H(X)}{\nu_H(XY)} \right)^{1/2} + \left(\frac{\mu_H(Y)}{\mu_H(XY)} \right)^{1/2} \leq 1. 
\]
\end{fact}

We are now ready to prove the main theorem of the section. 

\begin{theorem}\label{thm: PSL2}
Let $G=\widetilde{\PSL}_2(\mathbb R)$, let $\mu_G$ be a Haar measure on $G$, and let $X,Y\subseteq G$ be compact sets of positive measure. Then
\begin{equation}\label{eq: PSL2}
 \mu_G(XY)^{1/3}\geq \mu_G(X)^{1/3}+\mu_G(Y)^{1/3}.
\end{equation}
\end{theorem}

\begin{proof}
Let us first use the inequality on $H$ (Fact~\ref{fact: H}) to obtain an estimate on every transported product fibre. 

Fix $t\in\Omega$. By the construction of $\Omega$, we have $f(t)>0$, $0<p(t)<\infty$, and
\begin{equation}
  p(t)g(T(t))=\frac{\beta}{\alpha}f(t).
  \label{eq:main-transport-identity}
\end{equation}
In particular, we have $g(T(t))>0$. Hence $C_t$ has positive right Haar measure and $D_{T(t)}$ has positive left Haar measure, and $E_t=C_tD_{T(t)}$ has positive measure for both Haar measures. 

Define the two ratios
\begin{equation}
  r(t)=\left(\frac{f(t)}{\nu_H(E_t)}\right)^{1/2},
  \qquad\text{and}\qquad
  s(t)=\left(\frac{g(T(t))}{\mu_H(E_t)}\right)^{1/2}.
  \label{eq:main-fibre-ratios}
\end{equation}
Applying Fact~\ref{fact: H} to $C_t$ and $D_{T(t)}$ gives
$r(t)+s(t)\leq 1$. By H\"older's inequality we have
\begin{align*}
    f(t)^{1/3}+\bigl(p(t)g(T(t))\bigr)^{1/3}
  &=\nu_H(E_t)^{1/3}r(t)^{2/3}
    +\bigl(p(t)\mu_H(E_t)\bigr)^{1/3}s(t)^{2/3}\\
  &\leq
  \bigl(\nu_H(E_t)+p(t)\mu_H(E_t)\bigr)^{1/3}
  \bigl(r(t)+s(t)\bigr)^{2/3}\\
  &\leq \bigl(\nu_H(E_t)+p(t)\mu_H(E_t)\bigr)^{1/3}.
\end{align*}
Substituting \eqref{eq:main-transport-identity} into the left-hand side and cubing the resulting inequality, we obtain for $t\in\Omega$ 
\begin{equation}
  \nu_H(E_t)+p(t)\mu_H(E_t)
  \geq f(t)\left(1+\left(\frac{\beta}{\alpha}\right)^{1/3}\right)^3
  \label{eq:main-pointwise-estimate}
\end{equation}
Finally, by Proposition~\ref{prop: transported diagonal}
\begin{align*}
  \mu_G(XY)
  &\geq\int_\Omega
    \bigl(\nu_H(E_t)+p(t)\mu_H(E_t)\bigr)\d t\\
  &\geq\left(1+\left(\frac{\beta}{\alpha}\right)^{1/3}\right)^3
    \int_\Omega f(t)\d t =\bigl(\alpha^{1/3}+\beta^{1/3}\bigr)^3.
\end{align*}
Recall that $\alpha=\mu_G(X)$ and $\beta=\mu_G(Y)$, this proves the theorem. 
\end{proof}

When $X=Y$, Theorem~\ref{thm: PSL2} shows $\mu_G(X^2)\geq 8\mu_G(X)$.  It is worth noting how much of \eqref{eq: PSL2} is new. For a noncompact semisimple Lie group $L$ and compact $X\subseteq L$, it was shown in~\cite[Corollary~1.6]{JTZ} that $\mu_L(X^2)\ge4\mu_L(X)$, and for $L=G$ this is the best one can get from the exponent $1/(\ndim(G)-\hdim(G))=1/2$ available there. On the other hand, the
constant $8$ cannot be improved, by~\cite[Theorem~1.3]{JTZ} applied with $\ndim(G)=3$. 


\section{The Hermitian contact parabolic factorization}
\label{sec:contact}

In this section, $G$ will be a simply connected noncompact simple Lie group of Hermitian type other than $\widetilde{\PSL}_2(\R)$. Equivalently, the Lie algebra $\mathfrak{g}$ of $G$ occurs in the list~\eqref{eq:hermitian-list} and it is not $\mathfrak{sp}_2$. We will attach to $G$ a parabolic subgroup $P$ with the following two properties: the group $G$ is the product of $P$ with a compact subgroup, and $P$ has the same noncompact Lie dimension as $G$. 

We use various definitions and facts about Lie groups from~\cite{Knapp} and~\cite{HilgertNeeb}. Fix a Cartan decomposition $\mathfrak{g}=\mathfrak{k}\oplus\mathfrak{s}$. Being of Hermitian type means precisely that the compactly embedded subalgebra $\mathfrak{k}$ has a one-dimensional center, so that
\begin{equation}\label{eq:k-c-z}
  \mathfrak{k}=\mathfrak{c}\oplus\R Z,\qquad
  \mathfrak{c}=[\mathfrak{k},\mathfrak{k}]
\end{equation}
with $\mathfrak{c}$ compact semisimple. See~\cite[Chapter~VIII]{Helgason}. We let $K$ and $C$ be the analytic subgroups of $G$ with Lie algebras $\mathfrak{k}$ and $\mathfrak{c}$ respectively.

The following lemma tells us that, although the maximal compact subgroup of $G$ becomes a noncompact group when we pass to the universal cover, the semisimple part $\mathfrak{c}$ of $\mathfrak{k}$ still integrates to a compact subgroup, and this subgroup is the one computing the noncompact Lie dimension of $G$. 

\begin{lemma} \label{lem:C-maxcompact}
Suppose $G$, $K$, and $C$ are as above. Then we have the following:
\begin{enumerate}[(1)]
    \item $C$ is compact and connected, and $K$ is isomorphic as a Lie group to $C\times\R$;
    \item $G/C$ is diffeomorphic to a Euclidean space, and in particular it is contractible;
    \item $C$ is a maximal compact subgroup of $G$, and hence
    \begin{equation}\label{eq:nG-dimC}
      \ndim(G)=\dim G-\dim C.
    \end{equation}
\end{enumerate}
\end{lemma}

\begin{proof}
Let $G=KAN$ be an Iwasawa decomposition. Hence multiplication gives a diffeomorphism
\[
   K\times A\times N\to G,
\]
and $A \times N $ is contractible. Hence $K$ is homotopy equivalent to $G$. Since $G$ is simply connected, $K$ is simply connected. As we have
\[
   \mathfrak k=\mathfrak c\oplus\mathbb RZ,
\]
where $\mathfrak c$ is compact semisimple and $\mathbb RZ$ is central. Therefore by Fact \ref{fact:standard-structure}(1) we have
\[
   K\simeq C\times\mathbb R,
\]
where $C$ is compact and connected. This proves part~(1).

The Iwasawa diffeomorphism is equivariant under left multiplication by $C$. Consequently,
\[
   C\backslash G
   \simeq (C\backslash K)\times A\times N
   \simeq \mathbb R\times A\times N,
\]
which is a Euclidean space. Inversion $g\mapsto g^{-1}$ identifies $C\backslash G$ with $G/C$. Thus $G/C$ is Euclidean and in particular contractible. This proves part~(2).

Since $C$ is compact and $G/C$ is contractible, Fact \ref{fact:standard-structure}(2) shows that $C$ is a maximal compact subgroup of $G$. It follows directly from the definition of the noncompact Lie dimension that
\[
   n(G)=\dim G-\dim C.
\]
This proves part~(3).
\end{proof}

We now introduce the parabolic subgroups mentioned at the beginning of the section. Let $G_{\mathrm{ad}}$ be the adjoint group of $G$, and let $P_{\mathrm{ad}}$ be the contact parabolic subgroup of $G_{\mathrm{ad}}$, that is, the parabolic subgroup determined by the first Harish--Chandra strongly orthogonal root. Let $\pi: G\to G_{\mathrm{ad}}$ be the covering map, and set
\begin{equation}\label{eq:P-connected-lift}
  P=\bigl(\pi^{-1}(P_{\mathrm{ad}})\bigr)_0.
\end{equation}
Then $P$ is a connected closed subgroup of $G$ whose Lie algebra is the contact parabolic subalgebra $\mathfrak{p}$.

The next proposition is the main result of this section, and it is the structural reason that the three-dimensional base case controls all the Hermitian types.

\begin{proposition} \label{prop:contact-factorization}
Suppose $G$ is a simply connected noncompact simple Lie group of Hermitian type with $G\neq\widetilde{\PSL}_2(\R)$, $C$ is as in Lemma~\ref{lem:C-maxcompact}, the group $P$ is as in~\eqref{eq:P-connected-lift}, and $H=C\cap P$. Then we have the following:
\begin{enumerate}[(1)]
    \item $H$ is compact, $G=CP=PC$, and $\dim(P)<\dim(G)$.
    \item $H$ is a maximal compact subgroup of $P$, and
    \begin{equation}\label{eq:CP-and-dimension}
      \ndim(P)=\ndim(G).
    \end{equation}
\end{enumerate}
\end{proposition}

\begin{proof}
We first prove (1). As $C$ and $P$ are closed subgroups of $G$ and $C$ is compact, $H=C\cap P$ is a closed subgroup of the compact group $C$, hence compact.

It remains to show that $G=CP=PC$. The contact grading classification of \cite[Lemma~1.1, (1.5)--(1.8)]{ZhangHeisenberg} determines the Lie algebras $\mathfrak{c}\cap\mathfrak{p}$ and $[\mathfrak{m},\mathfrak{m}]$ for every $\mathfrak{g}$ in~\eqref{eq:hermitian-list}. These are recorded type by type in Appendix~\ref{app:classification}, and the dimension computation gives us
\begin{equation}\label{eq:dimension-orbit}
  \dim C-\dim H=\dim G/P .
\end{equation}
Note that the assumption $G\neq\widetilde{\PSL}_2(\R)$ guarantees $\mathfrak{c}\neq 0$, so $P$ is a proper contact parabolic subgroup, and $\dim G/P>0$.  

Let $C$ act on the connected manifold $G/P$ by left translations, and consider the orbit map $\varphi: C\to G/P$ given by $c\mapsto c\cdot eP$. The image of the differential of $\varphi$ at the identity is $\mathfrak{c}\cdot eP\subseteq T_{eP}(G/P)$, which has dimension $\dim C-\dim H$. By~\eqref{eq:dimension-orbit}, the latter number is $\dim G/P=\dim T_{eP}(G/P)$, so $\d\varphi_e$ is surjective. Composing with left translations, we get that $\d\varphi_c$ is surjective for every $c\in C$. Hence $\varphi$ is a submersion, and its image which is the orbit $C\cdot eP$ is open in $G/P$. On the other hand, this orbit is the continuous image of the compact group $C$, hence it is also compact. Therefore it is closed in the Hausdorff space $G/P$. A nonempty subset of the connected manifold $G/P$ that is simultaneously open and closed must be all of $G/P$. Therefore $C\cdot eP=G/P$, which means that every coset $gP$ meets $C$, that is $G=CP$. Taking inverses we also get $G=PC$. This completes the proof of (1).

We now prove (2). Using (1), we have $G=PC$, so every coset in $G/C$ is of the form $pcC=pC$ with $p\in P$. Hence $P$ acts transitively on $G/C$ by left translations, and the stabilizer of $eC$ under this action is $\{p\in P:pC=C\}=P\cap C=H$. Therefore the induced map
$P/H\to G/C$ with $pH\mapsto pC$ is a $P$-equivariant diffeomorphism from $P/H$ onto $G/C$. By Lemma~\ref{lem:C-maxcompact}(2), the latter is diffeomorphic to a Euclidean space and in particular contractible. Consequently, $P/H$ is contractible.

Recall from (1) that $H$ is compact. By Fact~\ref{fact:standard-structure}(2) again we get that $H$ is a maximal compact subgroup of $P$. Combining this with~\eqref{eq:nG-dimC} and the diffeomorphism $P/H\simeq G/C$ obtained above, we have
\[  \ndim(P)=\dim P-\dim H=\dim(P/H)=\dim(G/C) =\dim G-\dim C=\ndim(G), 
\]
which proves (2).
\end{proof}

\begin{table}[t]
\centering
\caption{The contact parabolic factorization.}\label{tab:contact-summary}
\small
\renewcommand{\arraystretch}{1.18}
\begin{tabular}{@{}p{.20\textwidth}p{.18\textwidth}p{.16\textwidth}p{.15\textwidth}p{.19\textwidth}@{}}
\toprule
$\mathfrak g$ & $\mathfrak c$ & $\dim G/P$ & $\ndim(G)$
\\
\midrule
$\mathfrak{su}(p,q)$, \qquad $(p,q)\neq(1,1)$
& $\mathfrak{su}(p)+\mathfrak{su}(q)$
& $2(p+q)-3$ & $2pq+1$ \\
$\mathfrak{sp}_{2n}(\R)$, $n\geq2$
& $\mathfrak{su}(n)$
& $2n-1$ & $n^2+n+1$ \\
$\mathfrak{so}^*(2n)$, $n\geq3$
& $\mathfrak{su}(n)$
& $4n-7$ & $n^2-n+1$ \\
$\mathfrak{so}(2,n)$, $n\geq3$
& $\mathfrak{so}(n)$
& $2n-3$ & $2n+1$\\
$\mathfrak e_{6(-14)}$
& $\mathfrak{so}(10)$
& $21$ & $33$\\
$\mathfrak e_{7(-25)}$
& compact $\mathfrak e_6$
& $33$ & $55$ \\
\bottomrule
\end{tabular}
\end{table}

It is worth pointing out where the assumption $G\neq\widetilde{\PSL}_2(\R)$ is used. For $\mathfrak{g}=\mathfrak{sl}_2(\R)$, the subalgebra $\mathfrak{k}$ is abelian, so $\mathfrak{c}=0$ and the group $C$ is trivial. The contact parabolic subgroup then degenerates to the Borel subgroup $AN$, and both parts (1) and (2) of Proposition~\ref{prop:contact-factorization} fail: the  variety $G/AN$ is noncompact, and $\ndim(AN)=2<3=\ndim(G)$. This is the case handled separately by Theorem~\ref{thm: PSL2}.

\section{Proof of the main theorem}\label{sec:induction}

In this section, we prove the main theorem. 
The strategy is an induction on the dimension argument, with Theorem~\ref{thm: PSL2} as the base case and the contact parabolic obtained in Section~\ref{sec:contact} providing the induction step.  

We are now ready to prove the main induction result of the paper.

\begin{theorem}\label{thm:simple}
Suppose $G$ is a simply connected simple real Lie group, and $n=\ndim(G)$. Then $G$ satisfies $\BM(\ndim(G))$ for all compact $X,Y\subset G$ with positive measure. 
\end{theorem}

\begin{proof}
We do induction on $\dim(G)$. If the center of $G$ is finite, then $G$ satisfies $\BM(\ndim(G))$ by the helix-free theorem of \cite{JTZ}. Hence we may assume that the center of $G$ is infinite. Then $G$ is of Hermitian type, so $\mathfrak g$ occurs in the list \eqref{eq:hermitian-list}. When $G=\widetilde{\PSL}_2(\R)$, the result is proven in
Theorem~\ref{thm: PSL2} with $\ndim(G)=3$. We may now assume that $G$ is not $\widetilde{\PSL}_2(\R)$. 

Proposition~\ref{prop:contact-factorization} provides a connected closed proper parabolic subgroup $P$ and a connected compact subgroup $C$ such that
$G=CP$, $C\cap P$ is compact, $\ndim(P)=\ndim(G)$, and $\dim(P)<\dim(G)$. The induction hypothesis, together with the dimension reduction in Fact~\ref{fact:jtz}(5), shows that $P$ satisfies $\BM(\ndim(P))$

Finally, apply the cocompact factor principle, Fact~\ref{fact:jtz}(4). Here $G$ is connected and unimodular, $C$ is connected and unimodular, $P$ is connected and closed, and $C\cap P$ is compact. We conclude that $G$ satisfies $\BM(\ndim(P))=\BM(\ndim(G))$. This completes the induction.
\end{proof}

Theorem~\ref{thm:all-lc} can then be derived from Theorem~\ref{thm:simple} and the reduction theorem (Fact~\ref{fact:jtz}(5)). 

\section{Applications}

Since \cite{JTZ} there are several direct applications appears. For example, An, Zhang, and the authors proved a product set inverse theorem~\cite{AJTZ}, and Hrushovski and Falno~\cite{HruFal} and Falno~\cite{Falno} use \cite{JTZ} for studying metric approximate groups and hyperdefinable groups. In particular,  Falno~\cite[Proposition 2.13]{Falno} showed that 
\[
\mathrm{Lrank}(G)\leq 12 (\log k)^2. 
\]
A direct application of Theorem~\ref{thm:all-lc} is able to improve the bound to
\[
\mathrm{Lrank}(G)\leq\frac{\lfloor3\log_k\rfloor(\lfloor3\log_k\rfloor+1)}2. 
\]

Machado \cite{Machado} applies \cite{JTZ} to an Archimedean internal window of an approximate lattice. He obtains genuine arithmetic complexity estimates controlling noncompact Archimedean places and the degree of the defining number field.

In this section we record a few geometric applications of the main theorem. 

\subsection{Parallel volume on homogeneous quotients}
Let $G$ be a connected unimodular Lie group, let $K<G$ be compact, and write $\M=G/K$, $o=K$, and $\pi:G\to \M$. Choose a $G$-invariant Riemannian metric on $\M$. Its volume is the quotient of Haar measure under $\pi$ up to normalization.
For a set $\Omega\subseteq \M$, write
\[
 \Omega_r=\{z\in \M:d(z,\Omega)\le r\},
 \qquad B_r(o)=\{z:d(z,o)\le r\}.
\]

\begin{proposition}
\label{prop:quotientBM}
Assume that $\ndim(G)=\dim(G/K)=n>0.$
Then every compact $\Omega\subseteq \M$ and every $r>0$ satisfy
\[
 \vol_\M(\Omega_r)^{1/n}
 \geq \vol_\M(\Omega)^{1/n}+\vol_\M(B_r(o))^{1/n}.
\]
\end{proposition}

\begin{proof}
Set $X=\pi^{-1}(\Omega)$, and $Y=\pi^{-1}(B_r(o)).$ The compactness of $K$ makes $\pi$ proper, so $X$ and $Y$ are compact. We claim that
\[
 XY=\pi^{-1}(\Omega_r).              
\]
Indeed, if $xK\in\Omega$ and $yK\in B_r(o)$, invariance of the metric gives
\[
 d(xyK,xK)=d(yK,o)\le r,
\]
and hence $xyK\in\Omega_r$.  Conversely, if $zK\in\Omega_r$, choose $xK\in\Omega$ with $d(zK,xK)\le r$. Then $y=x^{-1}z$ satisfies
$yK\in B_r(o)$ and $z=xy$. Weil's quotient formula supplies a constant $c_K>0$ such that
\[
 \mu_G(\pi^{-1}E)=c_K\vol_\M(E).
\]
Apply Theorem~\ref{thm:all-lc} to $X,Y$ and cancel $c_K^{1/n}$ get the desired result.
\end{proof}

Let us look at the following example. 
\begin{example}\label{eq: H2}
Let $G=\mathrm{PSL}_2(\R)$, $K=\mathrm{SO}(2)$ and
$\M=\mathbb H^2$ with curvature $-1$.  Since
\[
 \vol(B_t)=2\pi(\cosh t-1),
\]
the choice $\Omega=B_s(o)$ turns Proposition~\ref{prop:quotientBM} into
\[
(\cosh(s+r)-1)^{1/2}
 \geq (\cosh s-1)^{1/2}+(\cosh r-1)^{1/2}.
\]
The inequality is strict for $r,s>0$, but its ratio tends to one as $r,s\to 0$. 
\end{example}

In general, assume $\ndim (G)=\dim \M=n$.  For every $a,b>0$,
\[
  \lim_{t\to0}
  \frac{\vol_\M(B_{(a+b)t}(o))^{1/n}}
       {\vol_\M(B_{at}(o))^{1/n}
        +\vol_\M(B_{bt}(o))^{1/n}}
  =1.
\]
On the other hand, in real hyperbolic space, Erhard Schmidt~\cite{Schmidt49} proved the strictly stronger result that
\begin{equation}\label{eq:schmidt}
    \vol(\Omega_r) \geq \vol (B_{R+r})
\end{equation}
whenever $\vol(\Omega) = \vol(B_R)$. 

\subsection{Isoperimetry on symmetric spaces}

Let $\M=G/K$ be a symmetric space of noncompact type with its invariant Riemannian metric, and put $n=\dim M$.  We take $G$ to be the identity
component of the effective isometry group.  Then $G$ is semisimple, $K$ is maximal compact, and
\[
  \ndim(G)=\dim(G/K)=n.
\]
Throughout, $\omega_n$ denotes the Euclidean volume of the unit ball in $\mathbb R^n$.

\begin{proposition}[Cartan--Hadamard bound]
\label{prop :symmetric}
Let $\M$ be an $n$-dimensional symmetric space of noncompact type. For every nonempty finite-volume Borel set $E\subseteq \M$ and every $r>0$,
\begin{equation}
  \vol(E_r)
  \ge
  \left(\vol(E)^{1/n}+\omega_n^{1/n}r\right)^n.
  \label{eq:euclidean-parallel}
\end{equation}
Every finite-volume set $E\subseteq \M$ of finite perimeter satisfies
\begin{equation}
  \Per(E)
  \ge
  n\omega_n^{1/n}\vol(E)^{(n-1)/n}.
  \label{eq:numerical-CH}
\end{equation}
\end{proposition}

\begin{proof}
Apply Proposition~\ref{prop:quotientBM}. The radial Jacobi field comparison on a Cartan--Hadamard manifold gives
\[
  \vol(B_r(o))\ge\omega_n r^n,
\]
and substitution proves \eqref{eq:euclidean-parallel}.

First suppose that $\Omega$ is a relatively compact smooth domain. As $r\to 0$, the outer tube expansion is
\[
  \vol(\Omega_r)=\vol(\Omega)+r\Per(\Omega)+o(r).
\]
On the other hand, \eqref{eq:euclidean-parallel} gives
\[
  \frac{\vol(\Omega_r)-\vol(\Omega)}{r}
  \geq \frac{\bigl(\vol(\Omega)^{1/n}+\omega_n^{1/n}r\bigr)^n -\vol(\Omega)}{r}.
\]
Letting $r\to0$ yields
\[
  \Per(\Omega)\geq n\omega_n^{1/n}\vol(\Omega)^{(n-1)/n}.
\]

We record the approximation needed for a general finite-perimeter set. Choose a base point $o$.  The coarea formula supplies radii $R_j\to\infty$ such that the boundary contribution of $E\cap B_{R_j}(o)$ tends to zero.  More explicitly, for almost every $R$, the BV product rule gives
\[
  \Per(E\cap B_R)
  \le
  \Per(E;B_R)
  +\mathcal H^{n-1}(E^{(1)}\cap\partial B_R),
\]
and one may choose $R_j$ so that the last term tends to zero. Since $\chi_{E\cap B_{R_j}}\to\chi_E$ in $L^1$, lower semicontinuity and the displayed upper bound imply
\[
  \Per(E\cap B_{R_j})\to\Per(E).
\]
Local convolution in finitely many coordinate charts and a diagonal argument then produce relatively compact smooth domains $\Omega_j$ such that
\[
  \chi_{\Omega_j}\to\chi_E\text{ in }L^1(\M),
  \qquad\text{and}\qquad
  \Per(\Omega_j)\to\Per(E).
\]
This is the standard strict approximation theorem for BV functions on Riemannian manifolds. Applying the smooth inequality to $\Omega_j$ and passing to the limit proves \eqref{eq:numerical-CH} as desired.
\end{proof}

We remark that the above result is numerical, in the sense that it proves the Euclidean lower bound with its optimal homogeneous coefficient, but it does not identify isoperimetric regions, prove that geodesic balls minimize at a
fixed positive volume, or classify equality. On the other hand the result does not require any regularity assumptions. 

Let us now discuss the optimality.
\begin{proposition}[Optimal homogeneous coefficient]\label{prop:perimeter-sharp}
For every symmetric space $\M$ of noncompact type,
\[
  \inf_{0<\vol(E)<\infty}
  \frac{\Per(E)}{\vol(E)^{(n-1)/n}}  =n\omega_n^{1/n}.
\]
\end{proposition}

\begin{proof}
The lower bound is \eqref{eq:numerical-CH}.  Since $\M$ is locally symmetric, the small geodesic balls have expansions
\[
\begin{aligned}
  \vol(B_\varepsilon)
  &= \omega_n\varepsilon^n
  \left(  1-\frac{\mathrm{Scal}_\M}{6(n+2)}\varepsilon^2   +O(\varepsilon^4) \right),\\
  \Per(B_\varepsilon)
  &= n\omega_n\varepsilon^{n-1}
  \left(1-\frac{\mathrm{Scal}_\M}{6n}\varepsilon^2
    +O(\varepsilon^4) \right).
\end{aligned}
\]
Therefore
\begin{equation}
  \frac{\Per(B_\varepsilon)}
       {\vol(B_\varepsilon)^{(n-1)/n}}
  =
  n\omega_n^{1/n}
  \left(
    1-\frac{\mathrm{Scal}_\M}{2n(n+2)}\varepsilon^2
    +O(\varepsilon^4)
  \right).
  \label{eq:small-ball-ratio}
\end{equation}
The ratio tends to $n\omega_n^{1/n}$.  On a nonflat irreducible noncompact symmetric space, $\mathrm{Scal}_\M<0$, so it approaches the Euclidean constant from above.
\end{proof}

One should notice that at large volume, the above lower bound is not the correct scale. Wang computed the Cheeger constant $h(\M)$ of every symmetric space of noncompact type \cite{Wang2015}, and in restricted root notation it is the volume entropy, equivalently the norm of the mean-curvature vector of the horospheres. Hence
\[
  \Per(E)\geq h(M)\vol(E).
\]
The useful two-scale summary is
\begin{equation}
  \Per(E)\geq \max\left\{ n\omega_n^{1/n}\vol(E)^{(n-1)/n},  h(M)\vol(E)
  \right\}.
  \label{eq:two-scale}
\end{equation}
The first term has the optimal local coefficient and the second has the correct large volume homogeneity.  Their maximum still does not determine the intermediate volume behaviors.

Finally let us mentioned a few related work in the literature.
The Cartan--Hadamard inequality for arbitrary Cartan--Hadamard manifolds was established in dimension two by Weil~\cite{Weil1926}, dimension three by Kleiner~\cite{Kleiner1992}, and dimension four by Croke~\cite{Croke1984}. Ros's product theorem~\cite{Ros2005} shows that Euclidean domination is preserved under Cartesian products, so reducible examples assembled from already settled factors are not new at the level of \eqref{eq:numerical-CH}.

Nardulli and Osorio Acevedo \cite{Nardulli} proved sharp small volume inequalities on complete manifolds of strong bounded geometry, including a scalar curvature correction. Symmetric spaces meets their hypotheses. Silini \cite{Silini2025} proved the small volume ball result in every rank-one noncompact symmetric space, with quantitative stability. Thus the point of Proposition~\ref{prop :symmetric} in the unresolved higher-dimensional classes is its all volume nature, not its behavior near zero volume.

A very recent work of Ghomi prove the result under large-nullity and curvature-pinching hypotheses \cite{Ghomi2026Nullity,Ghomi2026Pinched}; these do not cover general irreducible higher rank symmetric spaces, which contain zero curvature planes. Li, Lin, and Xu \cite{LiLinXu2026} obtain a diameter-dependent isoperimetric estimate, with an exponential loss, for nonpositively curved symmetric spaces in ambient dimension at most seven. That result does not yield the diameter-free Euclidean constant in \eqref{eq:numerical-CH}.

One should note that for real hyperbolic space, Schmidt's theorem recalled in \eqref{eq:schmidt} is stronger at every positive scale.  In the other rank-one spaces, exact unrestricted profiles are much subtler.  Silini's work gives small volume or symmetry restricted results \cite{SiliniHopf,Silini2025}.

For genuine quaternionic hyperbolic spaces $\mathbf H_{\mathbb H}^{m}$, $m\ge2$, the Cayley plane $\mathbf H_{\mathbb O}^{2}$, and irreducible higher rank spaces such as $\mathrm{SL}_3(\R)/\mathrm{SO}(3)$, Proposition~\ref{prop :symmetric} is new.

\subsection{Analytic inequalities}

We record some analytic consequences. The rearrangement arguments used in the section are classical and the input is the numerical isoperimetric inequality.

Let $\M^n$ be a complete infinite volume Riemannian manifold satisfying
\begin{equation}
  \Per(E)\geq n\omega_n^{1/n} \vol(E)^{(n-1)/n},
  \label{eq:abstract-iso}
\end{equation}
for every bounded smooth set $E$.  By strict BV approximation it then holds for every finite-volume finite-perimeter set.

For a measurable function $u$ vanishing at infinity, let $u^*$ be the Euclidean symmetric decreasing rearrangement of $|u|$.  Thus $u^*$ is radial and nonincreasing on $\R^n$, and for $t>0$
\[
  |\{u^*>t\}| =\vol(\{|u|>t\}).
\]

\begin{proposition}[Euclidean P\'olya--Szeg\H{o} transfer]
\label{prop:PS}
Assume \eqref{eq:abstract-iso}. For every $u\in C_c^\infty(\M)$ and $1\le p<\infty$,
\begin{equation}
  \int_{\R^n}|\nabla u^*|^p\d x
  \leq \int_\M|\nabla u|^p\d\vol.
  \label{eq:PS}
\end{equation}
Moreover, for $1\le q\le\infty$,
\[
  \|u^*\|_{L^q(\R^n)}=\|u\|_{L^q(M)}.
\]
The $p=1$ statement extends to compactly supported BV functions, with the gradient replaced by total variation.
\end{proposition}

\begin{proof}
Replace $u$ by $|u|$, and set $E_t=\{u>t\}$ be the level set, and $ m(t)=\vol(E_t)$.  For almost every $t>0$, the value $t$ is regular and the coarea formula gives
\begin{equation}
  -m'(t) =
  \int_{\partial E_t}\frac{1}{|\nabla u|}\d A.
  \label{eq:distribution-derivative}
\end{equation}
If $p>1$, H\"older's inequality on $\partial E_t$ gives
\[
  \Per(E_t) \leq \left(
    \int_{\partial E_t}|\nabla u|^{p-1}\d A \right)^{1/p} \left( \int_{\partial E_t}\frac{1}{|\nabla u|}\d A\right)^{(p-1)/p}.
\]
Using \eqref{eq:distribution-derivative} and
\eqref{eq:abstract-iso}, we obtain
\begin{equation}
  \int_{\partial E_t}|\nabla u|^{p-1}\d A
  \geq \frac{\Per(E_t)^p}{(-m'(t))^{p-1}}
  \geq  \frac{(n\omega_n^{1/n} m(t)^{(n-1)/n})^p}{(-m'(t))^{p-1}}.
  \label{eq:level-lower}
\end{equation}
The level set $\{u^*>t\}$ is a Euclidean ball of volume $m(t)$. Its perimeter is exactly $n\omega_n^{1/n} m(t)^{(n-1)/n}$, and $|\nabla u^*|$ is constant almost everywhere along its boundary. Thus equality holds in the Euclidean analogue of the right side of \eqref{eq:level-lower}. Integrating in $t$ and using coarea again,
\begin{align*}
  \int_\M|\nabla u|^p\d\vol=
  \int_0^\infty
  \int_{\partial E_t}|\nabla u|^{p-1}\d A\d t \geq
  \int_{\R^n}|\nabla u^*|^p\d x.
\end{align*}

Note that for $p=1$, coarea directly gives
\[
  \int_\M|\nabla u|\d\vol= \int_0^\infty\Per(E_t)\d t \geq \int_{\R^n}|\nabla u^*|\d x.
\]
Plateau values may be handled by regularization of the distribution function. Their interiors contribute no gradient. Equimeasurability and the layer-cake formula give the $L^q$ identities. Strict smooth approximation and lower semicontinuity give the BV extension.
\end{proof}

This implication from isoperimetry to P\'olya--Szeg\H{o} is classical. The proof is included to make the hypotheses and constants explicit. We have the following corollary for Sobolev constants. 

\begin{corollary}
\label{cor:sobolev}
Let $n\geq2$.  Every $u\in C_c^\infty(\M)$ satisfies
\begin{equation}
  \|u\|_{L^{n/(n-1)}(\M)}
  \le
  \frac{1}{n\omega_n^{1/n}}
  \|\nabla u\|_{L^1(\M)}.
  \label{eq:L1-sobolev}
\end{equation}
If $1<p<n$, $p'=p/(p-1)$, and $p^*=np/(n-p)$, then
\begin{equation}
  \|u\|_{L^{p^*}(\M)}
  \le
  S_{n,p}\|\nabla u\|_{L^p(\M)},
  \label{eq:p-sobolev}
\end{equation}
where the Euclidean Aubin--Talenti constant is
\[
  S_{n,p}
  =
  \pi^{-1/2}n^{-1/p}
  \left(\frac{p-1}{n-p}\right)^{1/p'}
  \left[
    \frac{\Gamma(1+n/2)\Gamma(n)}
         {\Gamma(n/p)\Gamma(1+n/p')}
  \right]^{1/n}.
\]
\end{corollary}

\begin{proof}
Apply the sharp Euclidean $L^1$-Sobolev inequality and  the Aubin--Talenti inequality to $u^*$.  Equimeasurability preserves the norm on the left and Proposition \ref{prop:PS} does not increase the gradient norm. The Euclidean constants are due to Federer--Fleming~\cite{FF1960} at the endpoint and to Aubin and Talenti~\cite{Aubin1976,Talenti1976} for $1<p<n$.
\end{proof}

We record another corollary regarding the Faber--Krahn lower bound. No boundedness, smoothness, or connectedness of $\Omega$ is required.

\begin{corollary}
\label{cor:FK}
Let $\Omega\subset \M$ be any open set with $0<\vol(\Omega)<\infty$.  For $1<p<\infty$, define
\[
  \lambda_{1,p}^{D}(\Omega) =
  \inf_{0\ne u\in C_c^\infty(\Omega)}\frac{\int_\Omega|\nabla u|^p\d \vol}{\int_\Omega|u|^p\d \vol}.
\]
If $B_V\subset\R^n$ is a Euclidean ball with $|B_V|=V=\vol(\Omega)$, then
\begin{equation}
  \lambda_{1,p}^{D}(\Omega)\geq\lambda_{1,p}^{D}(B_V) = \lambda_{1,p}^{D}(B_1)  \left(\frac{\omega_n}{V}\right)^{p/n}.
  \label{eq:p-FK}
\end{equation}
\end{corollary}

\begin{proof}
For $u\in C_c^\infty(\Omega)$, Proposition \ref{prop:PS} yields
\[
  \frac{\int_\Omega|\nabla u|^p\d\vol}
       {\int_\Omega|u|^p\d\vol}
  \ge
  \frac{\int_{\R^n}|\nabla u^*|^p\d x}
       {\int_{\R^n}|u^*|^p\d x}.
\]
The support of $u^*$ is contained in a Euclidean ball of volume at most $V$. Extending by zero, regard $u^*$ as an element of $W^{1,p}_0(B_V)$.  Its Rayleigh quotient is at least $\lambda_{1,p}^{D}(B_V)$. Taking the infimum over $u$ proves \eqref{eq:p-FK}. The Euclidean scaling law gives the equality on the right, and the radial eigenfunction of a Euclidean ball gives
\eqref{eq:FK} as desired.
\end{proof}

For $p=2$, Corollary~\ref{cor:FK} implies
\begin{equation}
  \lambda_1^{D}(\Omega) \geq j_{n/2-1,1}^{\,2}  \left(\frac{\omega_n}{\vol(\Omega)}\right)^{2/n},
  \label{eq:FK}
\end{equation}
where $j_{\nu,1}$ is the first positive zero of $J_\nu$. On the other hand, Corollary \ref{cor:FK} does not say that a geodesic ball in $M$ minimizes $\lambda_1$ among domains of fixed volume.

Let us finally remark that the constants in \eqref{eq:L1-sobolev}, \eqref{eq:p-sobolev}, and \eqref{eq:p-FK} cannot be improved as homogeneous constants on $\M$. This statement does not assert that an optimizer exists.

\appendix
\section{The contact parabolic classification and dimension computations}\label{app:classification}

In this section, we gather the Lie algebra computations used in Proposition~\ref{prop:contact-factorization}. All of them come from the contact grading determined by the first Harish--Chandra strongly orthogonal root.  The root space formulas that we need are proved in~\cite[Sections~1.2--2.1]{ZhangHeisenberg}.
Throughout, all entries below are statements about Lie algebras.  

At the Lie algebra level, let $\mathfrak k$ be a
maximal compactly embedded subalgebra and set
$ \mathfrak c=[\mathfrak k,\mathfrak k]$. 
Let $\mathfrak p$ be the contact parabolic and write
$\mathfrak h=\mathfrak c\cap\mathfrak p.$ 
We now go through the Hermitian list type by type.

Let $\overline{\mathfrak u}$ denote the nilradical of the parabolic opposite to $\mathfrak p$. The contact grading gives the vector space decomposition
\[
  \mathfrak g=\overline{\mathfrak u}\oplus\mathfrak p.
\]
Consequently,
$
  \dim(G/P) =\dim(\mathfrak g/\mathfrak p) =\dim\overline{\mathfrak u}.
$
It therefore remains to verify, type by type, that
\[
  \dim\mathfrak c-\dim\mathfrak h
  =\dim(\mathfrak g/\mathfrak p).
\]
We use the conventions
$
  \mathfrak u(0)=\mathfrak{su}(1)=\mathfrak{so}(1)=0.
$

\subsection*{Type AIII}
Let $\mathfrak g=\mathfrak{su}(p,q)$ with $1\le p\le q$ and
$(p,q)\ne(1,1)$.  The maximal compactly embedded subalgebra is
$\mathfrak k=\mathfrak s\bigl(\mathfrak u(p)\oplus\mathfrak u(q)\bigr)$, so that
\[
 \mathfrak c=\mathfrak{su}(p)\oplus\mathfrak{su}(q).
\]
From $\dim\mathfrak{su}(k)=k^2-1$ we get $\dim\mathfrak c=p^2+q^2-2$, and with $\dim\mathfrak g=(p+q)^2-1$ the noncompact Lie dimension is
\[
 \ndim(G)=\dim\mathfrak g-\dim\mathfrak c
 =\bigl((p+q)^2-1\bigr)-\bigl(p^2+q^2-2\bigr)=2pq+1 .
\]
At the Lie algebra level, $\mathfrak h=\mathfrak c\cap\mathfrak p$ is the codimension one subalgebra of
$\mathfrak u(p-1)\oplus\mathfrak u(q-1)$ cut out, for a suitable embedding, by equality of the two trace coordinates, so that
$\dim\mathfrak h=(p-1)^2+(q-1)^2-1$.  Hence
\begin{align*}
\dim\mathfrak c-\dim\mathfrak h
 &=\bigl(p^2+q^2-2\bigr)-\bigl((p-1)^2+(q-1)^2-1\bigr)\\
 &=2(p+q)-3=\dim(\mathfrak g/\mathfrak p) .
\end{align*}

\subsection*{Type CI}
Let $\mathfrak g=\mathfrak{sp}_{2n}(\R)$ with $n\ge2$, so that $\dim\mathfrak g=n(2n+1)$ and $\mathfrak k=\mathfrak u(n)$.  Then
\[
 \mathfrak c=\mathfrak{su}(n),\qquad
 \mathfrak h=\mathfrak{su}(n-1).
\]
Using $\dim\mathfrak{su}(k)=k^2-1$, we have
\begin{align*}
\dim\mathfrak c-\dim\mathfrak h&=(n^2-1)-\bigl((n-1)^2-1\bigr)=2n-1=\dim(\mathfrak g/\mathfrak p),\\
 \ndim(G)&=n(2n+1)-(n^2-1)=n^2+n+1 .
\end{align*}

\subsection*{Type DIII}
Let $\mathfrak g=\mathfrak{so}^*(2n)$ with $n\ge3$, so that
$\dim\mathfrak g=n(2n-1)$ and $\mathfrak k=\mathfrak u(n)$.  Then
\[
 \mathfrak c=\mathfrak{su}(n),\qquad
 \mathfrak h=\mathfrak{su}(2)\oplus\mathfrak{su}(n-2).
\]
Then $\dim\mathfrak h=3+\bigl((n-2)^2-1\bigr)=n^2-4n+6$, and we have
\begin{align*}
 \dim\mathfrak c-\dim\mathfrak h&=(n^2-1)-(n^2-4n+6)=4n-7=\dim(\mathfrak g/\mathfrak p),\\
 \ndim(G)&=n(2n-1)-(n^2-1)=n^2-n+1 .
\end{align*}

\subsection*{Type BDI}
Let $\mathfrak g=\mathfrak{so}(2,n)$ with $n\ge3$, so that $\dim\mathfrak g=\tfrac12(n+2)(n+1)$.  Here the relevant parabolic is the stabilizer of a maximal isotropic two plane, as the isotropic line parabolic does not produce the required dimension identity. For the contact parabolic,
\[
 \mathfrak c=\mathfrak{so}(n),\qquad
 \mathfrak h=\mathfrak{so}(n-2).
\]
From $\dim\mathfrak{so}(k)=\tfrac12 k(k-1)$, we have
\begin{align*}
 \dim\mathfrak c-\dim\mathfrak h
 &=\frac{n(n-1)-(n-2)(n-3)}2=2n-3=\dim(\mathfrak g/\mathfrak p),\\
 \ndim(G)
 &=\frac{(n+2)(n+1)-n(n-1)}2=2n+1 .
\end{align*}

\subsection*{Types EIII and EVII}
For $\mathfrak g=\mathfrak e_{6(-14)}$ one has $\dim\mathfrak g=78$ and
\[
 \mathfrak c=\mathfrak{so}(10),\quad
 \mathfrak h=\mathfrak{su}(5).
\]
With $\dim\mathfrak{so}(10)=45$ and $\dim\mathfrak{su}(5)=24$, we have
\[
\dim\mathfrak c-\dim\mathfrak h=21=\dim(\mathfrak g/\mathfrak p),\qquad
 \ndim(G)=33.
\]

For $\mathfrak g=\mathfrak e_{7(-25)}$ one has $\dim\mathfrak g=133$ and
\[
 \mathfrak c=\mathfrak e_{6(-78)},\quad
 \mathfrak h=\mathfrak{so}(10).
\]
As $\dim\mathfrak e_6=78$ and $\dim\mathfrak{so}(10)=45$, we have
\[
\dim\mathfrak c-\dim\mathfrak h=33=\dim(\mathfrak g/\mathfrak p),\qquad \ndim(G)=55.
\]

Hence, in every type we have the equality
\[
 \dim\mathfrak c-\dim\mathfrak h
  =\dim(\mathfrak g/\mathfrak p).
\]

The exceptional case $\mathfrak g\simeq\mathfrak{sl}_2(\mathbb R)$ was treated separately in Theorem~\ref{thm: PSL2}. All remaining Hermitian cases the preceding dimension calculations establish the contact factorization.

\bibliographystyle{amsalpha}
\bibliography{reference}

\end{document}